\documentclass[12pt]{amsart}

\usepackage[utf8]{inputenc}
\usepackage[english]{babel}
\usepackage{amsmath}
\usepackage[noadjust]{cite}
\usepackage{amsthm}
\usepackage{amsfonts}
\usepackage{amssymb}
\usepackage{enumitem}
\usepackage{hyperref}
\usepackage{float}
\usepackage{graphicx}
\usepackage[margin=1.5in]{geometry}
\usepackage{etoolbox}
\apptocmd{\sloppy}{\hbadness 10000\relax}{}{}

\newtheorem{theorem}{Theorem}[subsection]
\newtheorem{theorem*}{Theorem}
\newtheorem{corollary}[theorem]{Corollary}
\newtheorem{corollary*}{Corollary}
\newtheorem{lemma}[theorem]{Lemma}
\newtheorem{prop}[theorem]{Proposition}

\theoremstyle{definition}
\newtheorem{definition}[theorem]{Definition}
\newtheorem{remark}[theorem]{Remark}

\binoppenalty=\maxdimen
\relpenalty=\maxdimen

\begin{document}
\title[The disk complex and topologically minimal surfaces]{The
disk complex and topologically minimal surfaces in the 3-sphere}
\author{Marion Campisi}
\address{San Jos\'{e} State University\\ San Jose, CA 95192}
\email{marion.campisi@sjsu.edu}
\author{Luis Torres}
\address{San Jos\'{e} State University\\ San Jose, CA 95192}
\email{luis.torres@sjsu.edu}
\maketitle

\begin{abstract}
We show that the disk complex of a genus $g>1$ Heegaard surface for the 3-sphere
is homotopy equivalent to a wedge of $(2g-2)$-dimensional spheres. This implies
that genus $g>1$ Heegaard surfaces for the 3-sphere are topologically minimal
with index $2g-1$.
\end{abstract}


\section{introduction}
Let $\Sigma$ be a compact, connected, orientable surface of genus $g>1$,
properly embedded in a compact, orientable 3-manifold $M$. The
\textit{disk complex} $\Gamma(\Sigma)$ is defined to be the simplicial complex
such that:

\begin{itemize}
\item vertices correspond to isotopy classes of essential curves that bound
compressing disks for $\Sigma$,
\item a set of $k$ vertices $\{v_1, \dots, v_k\}$ with mutually disjoint
representative curves in each isotopy class defines a $(k-1)$-simplex in
$\Gamma(\Sigma)$ with $v_1, \dots, v_k$ as its vertices. We denote the
$(k-1)$-simplex by $\langle \alpha_1, \dots, \alpha_k \rangle$, where $\alpha_i$
is a representative curve for the isotopy class corresponding to $v_i$.
\end{itemize}

In \cite{B}, Bachman introduced the \textit{topological index} of $\Sigma$ as
the smallest $n$ such that $\pi_{n-1}(\Gamma(\Sigma))$ is nontrivial when
$\Gamma(\Sigma)$ is non-empty or 0 if $\Gamma(\Sigma)$ is empty. If $\Sigma$ has
a well-defined topological index, i.e. either $\Gamma(\Sigma) = \emptyset$ or it
has a nontrivial homotopy group, then we say $\Sigma$ is
\textit{topologically minimal}.

Topologically minimal surfaces generalize two important classes of surfaces in
3-manifolds. Surfaces of index 0 are precisely those which are incompressible,
and surfaces of index 1 are those which are strongly irreducible. Moreover, a
number of well-known results for incompressible and strongly irreducible
surfaces generalize to topologically minimal surfaces. For example, given an
incompressible surface and a topologically minimal surface in an irreducible
3-manifold, Bachman \cite{B} showed that they may be isotoped so that their
intersection consists of loops which are essential on both surfaces.

After Bachman's initial work on topologically minimal surfaces, Bachman and
Johnson \cite{BJ} exhibited the first examples of topologically minimal surfaces
with arbitrarily high index. This was done by a construction involving gluing
copies of two-bridge link complements in chains along their boundary tori. This
established that higher index surfaces exist, as previously only surfaces of
indices 0, 1, and 2 were known. Lee generalized their result to show that
every 3-manifold containing an incompressible surface contains surfaces of
arbitrarily high topological index \cite{L1}. Campisi and Rathbun then showed
that there exist hyperbolic 3-manifolds containing surfaces of arbitrarily high
index \cite{CR}. Although these examples establish the existence of
topologically minimal surfaces in a variety of 3-manifolds, they are obtained
through highly specific and artificial constructions. The exception to this is a
result of Lee, which shows that bridge spheres for the unknot are topologically
minimal \cite{L2}, though their precise index is unknown.

The topology of the disk complex is not well-understood, which makes it
difficult to study its homotopy groups. This in turn makes it difficult to
identify topologically minimal surfaces and their indices. In this paper, we
adapt machinery of Harer \cite{H} and Ivanov \cite{I} used in the
study of the curve complex to the study of the disk complex in order to prove
Theorems 1 and 2 below.

Let $\Sigma$ be a compact, connected surface embedded in $S^3$ as a Heegaard
surface possibly with open disks removed and satisfying $-\chi(\Sigma) \geq 2$.
Denote the genus and number of boundary components of $\Sigma$ by $g$ and $b$,
respectively, and let
\[
d = \begin{cases}
2g-2 & \text{ if } g>0, b=0, \\
2g-3+b & \text{ if } g>0, b>0, \\
b-4 & \text{ if } g=0, b \geq 4.
\end{cases}
\]
Note that $d$ is well-defined and nonnegative since $-\chi(\Sigma) \geq 2$.

\begin{theorem*}
The disk complex $\Gamma(\Sigma)$ is $(d-1)$-connected.
\end{theorem*}

\begin{theorem*}
The disk complex $\Gamma(\Sigma)$ is homotopy equivalent to a space of dimension
$d$.
\end{theorem*}

Suppose now that $\Sigma$ is a genus $g$ Heegaard surface for $S^3$. Since every
curve bounding non-trivial compressing disks for $\Sigma$ is an essential simple
closed curve in $\Sigma$, we note that the disk complex $\Gamma(\Sigma)$ may be
considered as a subcomplex of the curve complex $\mathcal{C}(\Sigma)$ (with
$\Sigma$ considered as an abstract surface). A famous theorem of Harer in
\cite{H} states that $\mathcal{C}(\Sigma)$ is homotopy equivalent to a wedge of
infinitely many $(2g-2)$-spheres, and in 2015, Birman, Broaddus, and Menasco
\cite{BBM} explicitly exhibited a homologically nontrivial $(2g-2)$-sphere in
$\mathcal{C}(\Sigma)$. The curves used to determine the $(2g-2)$-sphere in
$\mathcal{C}(\Sigma)$ all bound compressing disks for $\Sigma \subset S^3$ when
$\Sigma$ is embedded as a Heegaard surface, so these curves determine a
homologically nontrivial $(2g-2)$-sphere in $\Gamma(\Sigma)$. By the Hurewicz
theorem, it follows that $\pi_{2g-2}(\Gamma(\Sigma)) \neq 0$. Combining Theorems
1 and 2 together with the fact that $\pi_{2g-2}(\Gamma(\Sigma)) \neq 0$, we
obtain the first main result of this paper.

\begin{theorem*}
If $\Sigma$ is a genus $g>1$ Heegaard surface for $S^3$, the disk complex
$\Gamma(\Sigma)$is homotopy equivalent to a nontrivial wedge of
$(2g-2)$-spheres.
\end{theorem*}

Bachman conjectured that there were no topologically minimal surfaces in
$S^3$ \cite{B}. This was shown to be false by Appel and Gabai in Appel's
senior thesis \cite{Appel}, although this result was not published or publicly
available. Our present work independently confirms that the conjecture is false.
Indeed, the second main result of this paper is an immediate corollary of
Theorem 3.

\begin{corollary*}
If $\Sigma$ is a genus $g>1$ Heegaard surface for $S^3$, then it is
topologically minimal with index $2g-1$.
\end{corollary*}

Here is a brief outline of this paper. In Section 2, we show that simplicial
maps into the disk complex correspond to families of smooth, non-degenerate,
real-valued maps, and vice-versa, which is based on the work of Ivanov \cite{I}
for the curve complex. These correspondences allow us to prove Theorem 1. In
Section 3, we investigate the homotopy type of the disk complex using
combinatorial methods based on the work of Harer \cite{H} and prove Theorem
2. Afterwards, we briefly review the work of Birman, Broaddus, and Menasco
\cite{BBM} that explicitly exhibits a topological essential $(2g-2)$-sphere in
the curve complex of the closed surface of genus $g$. These curves bound
compressing disks when the surface is considered as a genus $g$ Heegaard surface
$\Sigma$ for the 3-sphere. Thus, they determine a topological essential
$(2g-2)$-sphere in the disk complex of $\Sigma$, which completes the proof of
Theorem 3.


\section{Smooth functions and connectivity of the disk complex}


\subsection{Preliminaries}

Throughout this paper, we adopt the convention that whenever we consider a
surface embedded in $S^3$, we consider it embedded as a Heegaard surface
possibly with open disks removed. We will emphasize such an embedding by calling
it a \textit{standard embedding}. Moreover, we will always consider
a surface to be standardly embedded in $S^3$ when considering its disk complex.

Throughout Section 2, let $\Sigma$ denote a compact, connected, orientable
surface with $\chi(\Sigma) \leq -2$. The following smooth functions will be
fundamental in investigating the connectivity of the disk complex associated
to surfaces in the 3-sphere.

\begin{definition}
Let $f:\Sigma \rightarrow \mathbb{R}$ be a smooth function. If $\Sigma$ has
boundary, we will assume that $f(\partial S) = 0$, where 0 is a regular value of
$f$ and $f \geq 0$. A component $\alpha$ of $f^{-1}(A)$ is called
\textit{non-singular} if $\alpha$ does not contain a critical point of $f$.
\end{definition}

Note that the non-singular components of such smooth functions are simple closed
curves.

\begin{definition}
A smooth function $f: \Sigma \rightarrow \mathbb{R}$ is called
\textit{non-degenerate} if there is a non-trivial simple closed curve among the
non-singular components of its level sets.
\end{definition}

The rest of this section will be devoted to establishing a correspondence
between simplicial maps into $\Gamma(\Sigma)$ and families of smooth
non-degenerate functions. This correspondence will allow us to study simplicial
maps into $\Gamma(\Sigma)$ from the perspective of smooth functions. In
particular, the convexity of the space of smooth functions and the simplicial
approximation theorem will allow us to prove that any continuous mapping of a
sphere of small enough dimension into $\Gamma(\Sigma)$ is contractible. This in
turn proves that the homotopy groups of $\Gamma(\Sigma)$ up to a certain
dimension (which depends on $\chi(\Sigma)$) are trivial.


\subsection{Constructing a simplicial map from a family of smooth non-degenerate
functions}

Given a family $\{f_t: \Sigma \rightarrow \mathbb{R}\}_{t \in P}$ of smooth
non-degenerate functions, where $P$ is a parameter space, we show how to
construct a non-unique abstract simplicial complex $C$ and a simplicial map
$|C| \rightarrow \Gamma(\Sigma)$, where $|C|$ is the geometric realization of
$C$.

Since the functions $f_t$ are non-degenerate, for each $t \in P$ there is an
$A_t \in \mathbb{R}$ and a non-singular component $\alpha_t$ of $f_t^{-1}(A_t)$
which is a non-trivial simple closed curve. Perturb $f_t$ to a Morse function
$g_t$ with $g_t$ injective on its critical points and
$f_t^{-1}(A_t) = g_t^{-1}(A_t)$.

\begin{definition}
A \textit{handlebody filling} of $\Sigma$ is a pair $(H, \Sigma)$ such that $H$
is a 3-dimensional handlebody with $\partial H = \Sigma$.
\end{definition}

\begin{lemma}[\cite{G}, Lemma 8]
Given a surface $\Sigma$ and a Morse function $f: \Sigma \rightarrow \mathbb{R}$
which is injective on critical points, there is a handlebody filling
$(H,\Sigma)$ of $\Sigma$ with the property that every regular level component of
$f$ bounds a disk in $H$.
\end{lemma}

Let $\widehat{\Sigma}$ be the closed surface that results from capping off any
boundary components of $\Sigma$ with disks. Extending $g_t$ to a Morse function
for $\widehat{\Sigma}$ which is injective on critical points, by Lemma 2.2.2 we
obtain a handlebody filling $(H,\widehat{\Sigma})$ of $\widehat{\Sigma}$. Now
consider $H$ embedded in $S^3$ such that $\partial H$ is standardly embedded in
$S^3$. Identify $\widehat{\Sigma}$ with $\partial H \subset S^3$ and $\alpha_t$
with its image under this embedding. Then $\alpha_t$ bounds a nontrivial
compressing disk for $\widehat{\Sigma} \subset S^3$. Identifying $\Sigma$ with
the surface obtained by deleting from $\widehat{\Sigma} \subset S^3$ the images
of the capping disks that produced $\widehat{\Sigma}$, we see that $\alpha_t$
bounds a nontrivial compressing disk for $\Sigma \subset S^3$.

Observe now that if $u$ lies in a sufficiently small neighborhood $U_t$ of
$t \in P$, then one of the components of $f_u^{-1}(A_t)$ must be non-singular
and isotopic to $\alpha_t$. Denote this component by $\alpha_{t_u}$. The
neighborhoods $\{U_t \, : \, t \in P\}$ form a covering of
the space $P$. Let $\{U_t \, : \, t \in V\}$ be a subcover (possibly finite in
the case of $P$ compact) of this covering, and let $C$ be the nerve of this
subcover. That is, $C$ is an abstract simplicial complex whose vertices are
identified with the elements of $V$, and $t_0, \dots, t_n$ form an $n$-simplex
in $C$ if and only if $U_{t_0} \cap \cdots \cap U_{t_n} \neq \emptyset$. We now
show that the correspondence $t \mapsto \langle \alpha_t \rangle$ gives a
well-defined simplicial map $|C| \rightarrow \Gamma(\Sigma)$. Let $\sigma$ be a
simplex of $C$, and let $W$ be the set of elements of $V$ corresponding to the
vertices of $\sigma$. By definition, we have
$\displaystyle \bigcap_{t \in W} U_t \neq \emptyset$, so let
$\displaystyle u \in \bigcap_{t \in W} U_t$. Then
\begin{enumerate}[label={(\roman*)}]
\item for all $t \in W$, $\alpha_{t_u}$ is isotopic to $\alpha_t$, i.e.
$\langle \alpha_{t_u} \rangle = \langle \alpha_t \rangle$;
\item the simple closed curves in $\{\alpha_{t_u} \, : \, t \in W\}$ are
non-singular components of level sets of $f_u$, so any two of these simple
closed curves are either disjoint or isotopic.
\end{enumerate}
By (ii), $\{\langle \alpha_{t_u} \rangle \, : \, t \in W\}$ determines a
simplex in $\Gamma(\Sigma)$, and by (i), this simplex coincides with the image
$\{ \langle \alpha_t \rangle \, : \, t \in W\}$ of $W$. Identifying $C$ with its
geometric realization $|C|$, we see that the correspondence
$t \mapsto \langle \alpha_t \rangle$ indeed determines a well-defined simplicial
map $|C| \rightarrow \Gamma(\Sigma)$.

A simplicial map obtained by the above process is said to \textit{realize} the
family of smooth maps $\{f_t: \Sigma \rightarrow \mathbb{R} \}_{t \in P}$.

Note that we do not require all smooth maps in the family to
be non-degenerate. Indeed, in the construction above, we only needed the
functions $f_t$ for $t \in V$ to be non-degenerate. In Section 2.3, we
will show that every simplicial map realizes a family of smooth non-degenerate
maps.

\begin{remark}
Let $\{f_t: \Sigma \rightarrow \mathbb{R}\}_{t \in P}$ be a family of smooth
non-degenerate functions, $Q$ be a subset of $P$, and
$\{f_t\}_{t \in Q}$ be a subfamily of $\{f_t\}_{t \in P}$. Given a simplicial
map $F: |C_Q| \rightarrow \Gamma(S)$ that realizes $\{f_t\}_{t \in Q}$, $F$ can
be extended to a simplicial map $F': |C_P| \rightarrow \Gamma(S)$ such that
$F'$ realizes $\{f_t\}_{t \in P}$ and $F'$ realizes $\{f_t\}_{t \in Q}$
when restricted to $|C_Q|$. By the construction above, note that $C_Q$ is the
nerve of a covering $\{U_t\}_{t \in V}$ of $Q$, and
$|C_Q| \rightarrow \Gamma(\Sigma)$ is determined by a fixed choice of
non-singular components $\alpha_t$ of $f_t^{-1}(A_t)$, $t \in V$. Every $U_t$
has the form $U_t' \cap Q$, where $U_t'$ is an open set in $P$. Moreover, $U_t'$
may be chosen so that $f_u^{-1}(A_t)$ has a non-singular component isotopic to
$\alpha_t$ for both $u \in U_t$ and $u \in U_t'$. Let $C_P$ be the nerve
of any covering containing $\{U_t' \, : \, t \in V\}$. Now take any choice of
simple closed curves determining $F':|C_P| \rightarrow \Gamma(\Sigma)$ that
includes $\alpha_t$, $t \in V$.  Then $F'$ realizes
$\{f_t\}_{t \in P}$ and, by construction, realizes $\{f_t\}_{t \in Q}$ when
restricted to $|C_Q|$.
\end{remark}

\subsection{Constructing a family of smooth non-degenerate functions from
a simplicial map}

To establish a correspondence between simplicial maps into $\Gamma(\Sigma)$
and families of smooth non-degenerate functions, it remains to show that we may
construct a family of smooth non-degenerate functions from a simplicial map.
Given such a simplicial map, we first show it realizes a family of
smooth maps which are not necessarily non-degenerate.

\begin{lemma} Every simplicial map $h: X \rightarrow \Gamma(\Sigma)$ realizes a
family of smooth functions $\{f_t: \Sigma \rightarrow \mathbb{R}\}_{t \in X}$.
\end{lemma}

Ivanov (\S 1.3, \cite{I}) showed this in the case of simplicial maps into
the curve complex $\mathcal{C}(\Sigma)$. Since $\Gamma(\Sigma)$ may be
considered as a subcomplex of $\mathcal{C}(S)$, it is clear that it also holds
for simplicial maps into the disk complex. We present Ivanov's proof below.

\begin{proof}
Let $h: X \rightarrow \Gamma(\Sigma)$ be a simplicial map and let $V$ be the
vertex set of $X$. For any vertex $v \in V$, let $\text{St}_v$ denote the closed
star of $v$ in the first barycentric subdivision of $X$. The set of barycentric
stars $\{\text{St}_v \, : \, v \in V\}$ form a closed covering of $X$, and the
geometric realization of its nerve coincides with $X$ (see Lemma 2.11 of
\cite{A} for a proof).

If $U_v$ is a sufficiently small neighborhood of $\text{St}_v$ for each
$v \in V$, then the geometric realization of the nerve of the open covering
$\{U_v \, : \, v \in V\}$ also coincides with $X$.

We now construct the family $\{f_t: \Sigma \rightarrow \mathbb{R}\}_{t \in X}$.
Equipping $\Sigma$ with a (complete) Riemannian metric of constant curvature
with geodesic boundary, for every vertex of $\Gamma(\Sigma)$ there is a unique
geodesic representative in the corresponding isotopy class. Since $h$ is
simplicial, vertices of $X$ are mapped to vertices of $\Gamma(\Sigma)$, so each
$v \in V$ corresponds to a geodesic $\alpha_v$ under $h$. For each $v \in V$,
choose a closed neighborhood $N_v$ of $\alpha_v$ such that
$N_v \cap \partial \Sigma = \emptyset$ and $N_v \cap N_w = \emptyset$ when
$\alpha_v \cap \alpha_w = \emptyset$. For each $v \in V$ choose a smooth
function $g_v: N_v \rightarrow \mathbb{R}$ and a regular value $A_v$ of $g_v$
such that $g_v^{-1}(A_v) = \alpha_v$ and $g_v > 0$. Lastly, choose a closed
neighborhood $N_0$ of $\partial \Sigma$ such that $N_0 \cap N_v = \emptyset$ for
all $v \in V$, and choose a smooth function $g_0 : N_0 \rightarrow \mathbb{R}$
such that 0 is a regular value of $g_0$, $g_0^{-1}(0) = \partial \Sigma$, and
$g_0 \geq 0$. Letting
\[
K = (X \times N_0) \cup
\left( \bigcup_{v \in V} \text{St}_v \times N_v \right) \subset X \times \Sigma,
\]
define $G : K \rightarrow \mathbb{R}$ by
\[
G(t,x) = \begin{cases}
g_0(x) & \text{if} \; (t,x) \in X \times N_0, \\
g_v(x) & \text{if} \; (t,x) \in \text{St}_v \times N_v.
\end{cases}
\]
Observe that if $\text{St}_v \cap \text{St}_w \neq \emptyset$, then $\{v, w\}$
is a simplex in $X$ and so $\alpha_v \cap \alpha_w = \emptyset$ and
$N_v \cap N_w = \emptyset$. Therefore the sets $\text{St}_v \times N_v$ are
pairwise disjoint and do not intersect $X \times N_0$, so $G$ is well-defined.
Now extend $G$ to a smooth function $F: T \times \Sigma \rightarrow \mathbb{R}$.
For each $t \in X$, define $f_t: \Sigma \rightarrow \mathbb{R}$ by
$f_t(x) = F(t,x)$. Then $\{f_t\}_{t \in X}$ is a family
of smooth functions, and $h$ realizes $\{f_t\}_{t \in X}$ by construction.
\end{proof}

If the simplicial complex $X$ is of sufficiently small dimension (which will
be made precise below), we will show that every simplicial map
$h: X \rightarrow \Gamma(\Sigma)$ also realizes a family of smooth
non-degenerate functions $\{f_t\}_{t \in X}$. The following theorem of Ivanov
will be instrumental.

\begin{theorem}[\cite{I}, Theorem 2.5]
Let
\[
d = \begin{cases}
-\chi(\Sigma) & \text{ if } \Sigma \text{ is closed,} \\
-\chi(\Sigma)-1 & \text{ if } \Sigma \text{ has exactly 1 boundary component,}\\
-\chi(\Sigma)-2 & \text{ otherwise.}
\end{cases}
\]
Then any family of smooth functions
$\{f_t:\Sigma \rightarrow \mathbb{R}\}_{t \in P}$ with $\dim P \leq d$ can be
approximated arbitrarily well by a family of smooth non-degenerate functions.
\end{theorem}

Finally, using Theorem 2.3.2 we establish that a simplicial map actually
realizes a family of smooth non-degenerate functions as desired. In Section 2.4,
this result will serve a key role in determining that
$\Gamma(\Sigma)$ is $(d-1)$-connected, i.e.
\[
\pi_i(\Gamma(\Sigma)) \cong 0, \quad 1 \leq i \leq d-1.
\]

\begin{lemma}
Let
\[
d = \begin{cases}
-\chi(\Sigma) & \text{ if } \Sigma \text{ is closed,} \\
-\chi(\Sigma)-1 & \text{ if } \Sigma \text{ has exactly 1 boundary component,}\\
-\chi(\Sigma)-2 & \text{ otherwise.}
\end{cases}
\]
Every simplicial map $h: X \rightarrow \Gamma(\Sigma)$ with $\dim X \leq d$
realizes a family of smooth non-degenerate functions
$\{f_t: \Sigma \rightarrow \mathbb{R}\}_{t \in X}$.
\end{lemma}

\begin{proof}
By Lemma 2.3.1, any simplicial map $h: X \rightarrow \Gamma(\Sigma)$ realizes a
family of smooth functions $\{f_t: \Sigma \rightarrow \mathbb{R}\}_{t \in X}$.
Let $\{\alpha_v: v \in V\}$ be the simple closed curves described in the proof
of Lemma 2.3.1. By Theorem 2.3.2,
$\{f_t: \Sigma \rightarrow \mathbb{R}\}_{t \in X}$ can be approximated
arbitrarily well by a family of smooth non-degenerate functions
$\{f_t'\}_{t \in X}$. Recalling the construction described in Section 2.2, we
will repeat the arguments and show how to obtain $h$ from such an approximation
$\{f_t'\}_{t \in X}$ via this construction. This will prove that $h$ realizes a
family of smooth non-degenerate functions.

Take sufficiently small neighborhoods $U_v'\subset U_v$ of $\text{St}_v$ for
each $v \in V$ together with a good enough approximation $\{f_t'\}_{t \in X}$
such that, for each $v \in V$, there is a regular value
$A'_t \in [A_t-\varepsilon, A_t+\varepsilon]$ for some $\varepsilon > 0$ such
that $(f_t')^{-1}(A_t')$ contains a non-singular component $\alpha_t'$ which is
isotopic to $\alpha_t$. Following the rest of the construction, the
correspondence $t \mapsto \langle \alpha_t' \rangle$ obtained is exactly the
same as the correspondence $t \mapsto \langle \alpha_t \rangle$ since
$\langle \alpha_t \rangle = \langle \alpha_t' \rangle$, so
$t \mapsto \langle \alpha_t' \rangle$ determines a mapping
$X \rightarrow \Gamma(\Sigma)$ that coincides with $h$. Therefore $h$ realizes
$\{f_t'\}_{t \in X}$.
\end{proof}

\subsection{Connectivity of the disk complex}

We are now ready to prove the main result of Section 2.

\begin{theorem}
Let
\[
d = \begin{cases}
-\chi(\Sigma) & \text{ if } \Sigma \text{ is closed,} \\
-\chi(\Sigma)-1 & \text{ if } \Sigma \text{ has exactly 1 boundary component,}\\
-\chi(\Sigma)-2 & \text{ otherwise.}
\end{cases}
\]
Then $\Gamma(\Sigma)$ is $(d-1)$-connected.
\end{theorem}

\begin{proof}
Let $N \leq d-1$. It suffices to show that any continuous map $S^N \rightarrow \Gamma(\Sigma)$ is homotopic to a constant map. In this direction, let
$g: S^N \rightarrow \Gamma(\Sigma)$ be a continuous map. By the simplicial
approximation theorem, there exists a triangulation $T$ of $S^N$ with
homeomorphism $\varphi: S^N \rightarrow T$ and a simplicial map
$h: T \rightarrow \Gamma(\Sigma)$ such that $g$ is homotopic to
$|h| = h \circ \varphi$. By Lemma 2.3.1, $h$ realizes a family of smooth
functions $\{f_t\}_{t \in T}$. Since the space of smooth functions is
contractible, $\{f_t\}_{t \in T}$ can be extended to a family of smooth
functions  $\{f_t\}_{t \in B}$, where $B$ is a simplicial $(N+1)$-ball with
boundary $T$. We have $\dim B \leq d$, so Theorem 2.3.2 shows that
$\{f_t\}_{t \in B}$ can be approximated arbitrarily well by a family of smooth
non-degenerate functions $\{f_t'\}_{t \in B}$. Note that the proof of
Lemma 2.3.3 shows that $h$ realizes $\{f_t'\}_{t \in T}$ if $\{f_t'\}_{t \in B}$
is a good enough approximation to $\{f_t\}_{t \in B}$. Taking such an
approximation $\{f_t'\}_{t \in B}$, Remark 2.2.4 shows that $h$ can be extended
to a simplicial map $h': B \rightarrow \Gamma(\Sigma)$. Therefore $h$ and thus
also $g$ is homotopic to a constant map.
\end{proof}


\section{The homotopy type of the disk complex}

Let $\Sigma$ be a compact, connected surface of genus $g>1$ with $b$ boundary
components, standardly embedded in $S^3$. A
subcomplex of either the curve complex or disk complex has
dimension at most $3g-4+b$ since a pants decomposition has at most $3g-3+b$
curves. In Section 3.1, we show that we can do better than this in terms of
homotopy dimension. Explicitly, we show that $\Gamma(\Sigma)$ is homotopy
equivalent to a space of dimension
\[
d = \begin{cases}
2g-2 & \text{ if } g>0, b=0, \\
2g-3+b & \text{ if } g>0, b>0, \\
b-4 & \text{ if } g=0, b \geq 4.
\end{cases}
\]
In Section 3.2, we use the work of Birman, Broaddus, and Menasco \cite{BBM} to
show that the disk complex $\Gamma(\Sigma)$ contains an essential
$(2g-2)$-sphere in the case where $\Sigma$ is a genus $g > 1$ Heegaard surface
for the 3-sphere. Combining these results with Theorem 2.4.1 gives the two main
results of this paper, which we collect in Section 3.3.

\subsection{Homotopy dimension of the disk complex}
In Theorem 3.5 of \cite{H}, Harer shows that the curve complex is homotopy
equivalent to a wedge of spheres of the same dimension as in each of the cases
of Theorem 3.1.1 below. Employing the techniques used in his proof, we obtain
the following result. For convenience, let $\Sigma_g^b$ denote a compact,
connected surface of genus $g$ with $b$ boundary components, standardly embedded
in $S^3$.

\begin{theorem}
The disk complex $\Gamma(\Sigma_g^b)$ is homotopy equivalent to a space of
dimension
\[
d = \begin{cases}
2g-2 & \text{ if } g>0, b=0, \\
2g-3+b & \text{ if } g>0, b>0, \\
b-4 & \text{ if } g=0, b \geq 4.
\end{cases}
\]
and is empty otherwise.
\end{theorem}

\begin{proof}
Assume first that $g=0$. Note that a sphere with one, two, or three boundary
components does not admit any essential simple closed curves, so
$\Gamma(\Sigma_0^b)$ is empty for $b<4$. For $b \geq 4$, a pants decomposition
of $\Sigma_0^b$ consists of $b-3$ curves, and such a set of curves gives a
$(b-4)$-simplex in $\Gamma(\Sigma_0^b)$ which is of maximal dimension.

Assume now that $g>0$. We will proceed by double induction over $g$ (outer
induction) and $b$ (inner induction), so assume that the theorem has been proven
for all surfaces of genus smaller than $g$ with arbitrarily many boundary
components.

Let $\partial \Sigma_g^b = \{P_1, \dots, P_b\}$ be the boundary components of
$\Sigma_g^b$. Define $\widehat{\Gamma}(\Sigma_g^b)$ to be the subcomplex of
$\Gamma(\Sigma_g^b)$ consisting of simplices
$\langle \alpha_0, \dots, \alpha_k \rangle$ for which no curve $\alpha_i$
bounds a region on $\Sigma_g^b$ which contains $P_1$ and exactly one other
boundary component $P_j$. Note that $\widehat{\Gamma}(\Sigma_g^b) =
\Gamma(\Sigma_g^b)$ for $b=1$. Attaching a disk to $P_1$ defines a map
$\Phi: \widehat{\Gamma}(\Sigma_g^b) \rightarrow \Gamma(\Sigma_g^{b-1})$.

\begin{lemma}
If $b>0$, the map $\Phi$ is a homotopy equivalence.
\end{lemma}

\begin{proof}
Let $F_0$ be the surface obtained from $\Sigma_g^b$ by attaching a disk to
$P_1$. Choose a hyperbolic metric for $F_0$ so that every simple closed curve in
$F_0$ is represented by a unique geodesic. There is at least one point in $F_0$
and a neighborhood of that point through which no simple geodesic passes.
Selecting this point to lie on the disk attached to $P_1$ and the neighborhood
to contain the disk defines a map
$\Psi: \Gamma(\Sigma_g^{b-1}) \rightarrow \widehat{\Gamma}(\Sigma_g^b)$. Note
that $\Phi \circ \Psi$ is the identity on $\Gamma(\Sigma_g^{b-1})$. The proof
will be complete once we show that the homomorphisms induced by $\Phi$ on
homotopy groups are injective.

Let $T$ be a simplicial $n$-sphere and let
$f: T \rightarrow \widehat{\Gamma}(\Sigma_g^b)$ be a simplicial map. We will
show that $f$ is homotopic to $f' = \Psi \circ \Phi \circ f$. Let
$v_1, \dots, v_k$ be the vertices of $T$, and write
$\langle \alpha_i \rangle = f(v_i)$ and $\langle \alpha_i' \rangle = f'(v_i)$,
where $\alpha_i$ and $\alpha_i'$ are the geodesic representatives of each
isotopy class. Each $\alpha_i$ is isotopic to $\alpha_i'$ in $F_0$, and if
$\alpha_i'$ and $\alpha_j'$ are isotopic in $F_0$ then they are already equal in
$\Sigma_g^b$ by construction.

For all $i$, the isotopy of $\alpha_i$ to $\alpha_i'$ describes a properly
embedded annulus $A_i$ in $F_0 \times [0,1]$ with
$A_i \cap (F_0 \times 0) = \alpha_i$ and $A_i \cap (F_0 \times 1) = \alpha_i'$.
We may arrange the annuli so that they are pairwise transverse and that each one
is transverse to $P_1 \times [0,1]$. However, it is necessary to arrange the
annuli so that if $\alpha_i$ and $\alpha_j$ are disjoint, then either $A_i$ and
$A_j$ are disjoint or $A_i \cap A_j = \alpha_i' = \alpha_j'$.

To do this, consider $A_i$ and $A_j$ with $\alpha_i$ and $\alpha_j$ disjoint.
The intersection $A_i \cap A_j$ is a properly embedded submanifold in
both $A_i$ and $A_j$. At one end, $A_i$ and $A_j$ are disjoint and at the other
end they are either equal or disjoint, hence any unwanted intersections between
them are circles. Circles of intersection which are inessential can be removed
by an innermost circle argument. Without loss of generality, suppose now that
$\alpha$ is a circle of intersection that is nontrivial in $\pi_1(A_i)$. Since
$\pi_1(F_0 \times [0,1]) = \pi_1(F_0)$, then $\alpha$ must also be nontrivial in
$\pi_1(A_j)$, otherwise it would be an inessential circle in $F_0$. Hence
$\alpha_i' = \alpha_j'$. Now interchange the portion of $A_i$ containing
$\alpha_i'$ with the portion of $A_j$ containing $\alpha_j'$ so that the
intersection at $\alpha$ is eliminated. Continue in this way until all essential
circles of intersection are removed.

After we have removed these intersection circles, the annuli are in the
desired positions. Move the $A_i$ slightly (if necessary) so that $P_1 \times
[0,1]$ intersects only one annulus at a time, and let
$0 < t_1 < \cdots < t_m < 1$ be the times that these intersections occur.
Suppose $A_i$ meets $P_1$ at time $t_1$. Let $\beta_i$ be the curve
$A_i \cap (F_0 \times (t_1 + \epsilon))$, where $\epsilon > 0$ is chosen so that
$t_1 + \epsilon < t_2$. By construction, $\beta_i$ can be isotoped to be
disjoint from $\alpha_i$, and the positioning of the annuli guarantee that
if $\alpha_i$ is disjoint from $\alpha_j$ then so must $\beta_i$. If $\Lambda =
\langle v_i, v_{j_1}, \dots, v_{j_n} \rangle$ is a face of $T$ in the star
of $v_i$, then add the simplex
$\langle v_i, \beta_i, v_{j_1}, \dots, v_{j_n} \rangle$ to $T$ along
$\Lambda$. The map $f$ is then extended by taking $\beta_i$ to
$\langle \alpha_i \rangle$. Continuing in this way, we add simplices and extend
$f$ for each of the intersections at each time $t_i$ as we did for $t_1$.
Completing this process gives the desired homotopy of $f$ to $f'$.
\end{proof}

Returning to the proof of Theorem 3.1.1, consider the first barycentric
subdivision, $\Gamma^0(\Sigma_g^b)$, of $\Gamma(\Sigma_g^b)$, which is defined
to be the simplicial complex such that

\begin{enumerate}[label=(\roman*)]
\item $\Gamma^0(\Sigma_g^b)$ has a vertex $v$ for every simplex in
$\Gamma(\Sigma_g^b)$, where the \textit{weight} of $v$ is defined as the
dimension of the simplex it represents,
\item a chain of $t+1$ proper inclusions of simplices in $\Gamma(\Sigma_g^b)$
defines a $t$-simplex in $\Gamma^0(\Sigma_g^b)$ with its vertices being those
that represent the simplices in the chain.
\end{enumerate}

Define $X_k$ to be the subcomplex of $\Gamma^0(\Sigma_g^b)$ whose vertices have
weight greater than or equal to $k$.  We will construct $\Gamma^0(\Sigma_g^b)$
starting from $X_{3g-4}$ by adding vertices of decreasing weight together with
the simplices they span in each step.

The subcomplex $X_{3g-4}$ is just a discrete set of points, so it has dimension
less than $2g-2$. Suppose that $X_{k+1}$ has been shown to be homotopy
equivalent to a space of dimension no more than $2g-2$. Consider a vertex
$v = \langle \alpha_0, \alpha_1, \dots, \alpha_k \rangle$ of $X_k - X_{k+1}$. If
$F_1, \dots, F_t$ are the components of $\Sigma_g^b$ obtained by splitting along
$\alpha_0, \dots, \alpha_k$, then the link of $v$ in $X_{k+1}$ is readily
identified with the join of the complexes $\Gamma^0(F_1), \dots, \Gamma^0(F_t)$.
Supposing that, for $i=1, \dots, t$, each component $F_i$ has genus $g_i$ and
$b_i$ boundary components, observe that $\chi(F_i) = 2-2g_i-b_i$. Note that
$g_i < g$ for each $i=1, \dots, t$ and that $\sum_{i=1}^t b_i = 2k+2$.

Assume now that $b=0$. Then we have
\begin{align*}
\chi(\Sigma_g^b) &= \sum_{i=1}^t \chi(F_i) \\
&= \sum_{i=1}^t (2 - b_i - 2g_i) \\
&= 2t - 2k - 2 - \sum_{i=1}^t 2g_i
\end{align*}
and hence
\[
g = k - t + 2 +\sum_{i=1}^t g_i.
\]
Inductively, for all $i$, $\Gamma^0(F_i)$ is homotopy equivalent to a space of
dimension $2g_i + b_i - 3$, so the link of $v$ in $X_{k+1}$ is homotopy
equivalent to a space of dimension
\begin{align*}
(t-1) + \sum_{i=1}^t (2g_i + b_i - 3) = 2g - 3.
\end{align*}

The join of $v$ with its link in $X_{k+1}$ is thus homotopy equivalent to a
space of dimension $2g-2$. To finish the proof for $b=0$, observe that we obtain
$X_k$ by adding the join of every vertex in $X_k-X_{k+1}$ with its link in
$X_{k+1}$, so $X_k$ is indeed homotopy equivalent to a space of dimension
$2g-2$. By this inductive construction of $\Gamma(\Sigma_g^b)$, it follows that
$\Gamma(\Sigma_g^b)$ is homotopy equivalent to a space of dimension $2g-2$, so
the Theorem holds for $b=0$. By Lemma 3.1.2, it also holds for $b=1$.

Now assume $b>2$, and suppose that $\Gamma(\Sigma_g^{b-1})$ is homotopy
equivalent to a space of dimension $2g-3+(b-1) = 2g-4+b$. Lemma 3.1.2 shows that
$\widehat{\Gamma}(\Sigma_g^b)$ is homotopy equivalent to the same
space. Now $\Gamma(\Sigma_g^b)$ is obtained from $\widehat{\Gamma}(\Sigma_g^b)$
by adding the stars of the omitted vertices along their links: the curves
corresponding to these vertices each bound regions in $\Sigma_g^b$ containing
$P_1$ and exactly one other boundary component, and if two such curves are
disjoint then they are isotopic. The map $\Phi$ sends the link of such a vertex
homomorphically onto $\Gamma(\Sigma_g^{b-1})$. Therefore $\Gamma(\Sigma_g^b)$
has the homotopy type of the join of $\Gamma(\Sigma_g^{b-1})$ with a set of
points, and hence is homotopy equivalent a space of dimension $2g-3+b$.
\end{proof}


\subsection{An essential sphere in the disk complex of a Heegaard surface for
the 3-sphere}

We now briefly review the work of Birman, Broaddus, and Menasco \cite{BBM} that
describes a subcomplex $X(\Sigma_g^0)$ of the curve complex
$\mathcal{C}(\Sigma_g^0)$ which is topologically an essential $(2g-2)$-sphere in
$\mathcal{C}(\Sigma_g^0)$. Considering $\Sigma_g^0$ as a Heegaard surface for
the 3-sphere, the vertices of $X(\Sigma_g^0)$ correspond to isotopy classes of
curves that bound compressing disks for $\Sigma_g^0 \subset S^3$. Since
$X(\Sigma_g^0)$ is defined as the span of these vertices, it follows that
$X(\Sigma_g^0)$ is a subcomplex of $\Gamma(\Sigma_g^0)$. This subcomplex will
allow us to prove that $\pi_{2g-2}(\Sigma_g^0) \not \cong 0$.

We will adopt the conventions of \cite{BBM}. We first
consider the surface with one boundary component $\Sigma_g^1$ to be standardly
embedded in $S^3$. Let $z_1', z_2', \dots, z_{2g-2}'$ denote the curves on
$\Sigma_g^1$ in Figure 1 below.

\begin{figure}[H]
\centering
\includegraphics[width=0.8\textwidth]{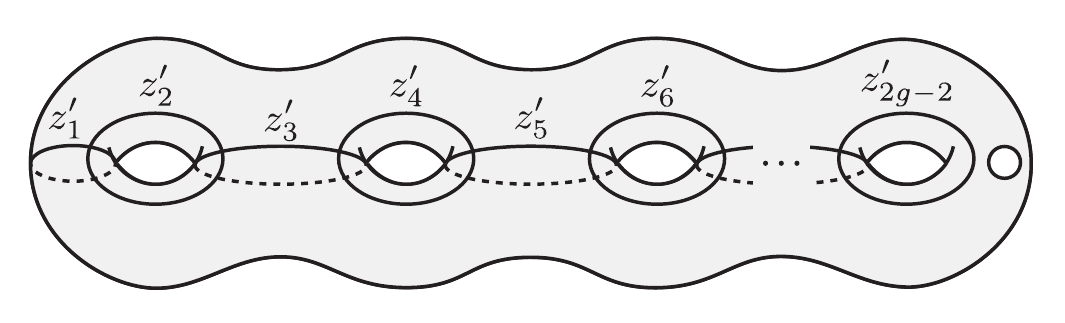}
\caption{The surface $\Sigma_g^1$ in $S^3$ with curves
$z'_1, z'_2, \dots, z'_{2g-2}$.}
\end{figure}

For every nonempty proper subset
$J = \{j, j+1, \dots, m\} \subsetneq \{1, 2, \dots, 2g\}$, let $N_J$ be a closed
regular neighborhood of $z'_j \cup z'_{j+1} \cup \cdots \cup z'_m$ in
$\Sigma_g^1$. If $|J|$ is even, then $N_J$ has a single boundary component,
which we denote by $x_J$ (see Figure 2).

\begin{figure}[H]
\centering
\includegraphics[width=0.8\textwidth]{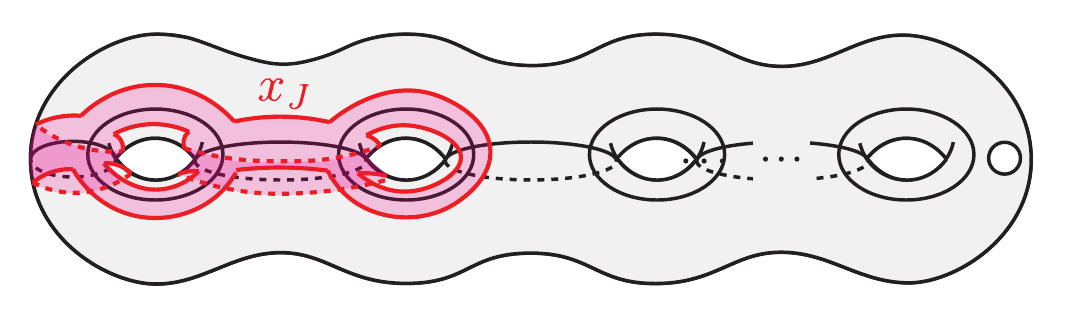}
\caption{The surface $\Sigma_g^1$ in $S^3$ and a typical curve $x_J$ in the case
where $|J|$ is even. Here, $J=\{1,2,3,4\}$ and $x_J$ is the boundary of a closed neighborhood $N_J$ of $z'_1 \cup z'_2 \cup z'_3 \cup z'_4$.}
\end{figure}

If $|J|$ is odd, then $N_J$ has two boundary components, and we let $x_J$ be the
boundary component of $N_J$ such that $x_J$ is in the ``back half'' of
$\Sigma_g^1$ if $j$ is even or $x_J$ is in the ``top half'' of $\Sigma_g^1$ if
$j$ is odd (see Figure 3).

\begin{figure}[H]
\centering
\includegraphics[width=1\textwidth]{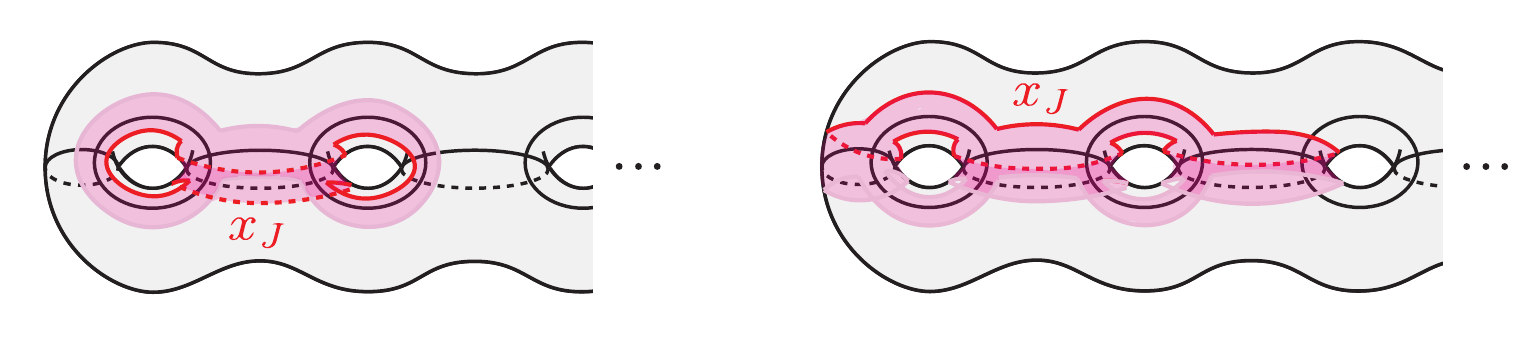}
\caption{(left) A typical curve $x_J$ in the case where $|J|$ is odd and $j$ is
even. (right) A typical curve $x_J$ in the case where $|J|$ is odd and $j$ is
odd.}
\end{figure}

Now let $X^0$ be the set
\[
X^0 = \{x_J : J =\{j, j+1, \dots, m\} \subsetneq \{1, 2, \dots, 2g\},
J \neq \emptyset\}
\]
and define $X(\Sigma_g^1)$ to be the subcomplex of all simplices of
$\mathcal{C}(\Sigma_g^1)$ whose vertices are in $X^0$. Notice that the vertices
in $X^0$ all correspond to isotopy classes of curves that bound compressing
disks for $\Sigma_g^1 \subset S^3$. Indeed, each curve $z_i'$ bounds a
compressing disk, so each curve $x_J$ (where $J = \{j, j+1, \dots, m\}$) is
isotopic to the boundary of a compressing disk obtained by
starting with either:
\begin{enumerate}[label={(\roman*)}]
\item two copies of each of the compressing disks bounded by the curves $z'_i$
with odd $i$ when $|J|$ is even,
\item one copy of each of the compressing disks bounded by the curves $z'_i$
with even $i$ when $|J|$ is odd and $j$ is even, or
\item one copy of each of the compressing disks bounded by the curves $z'_i$
with odd $i$ when $|J|$ is odd and $j$ is odd,
\end{enumerate}
and then banding these compressing disks along arcs lying on the remaining
$z'_i$ (see Figures 2, 3 (left), and 3 (right), respectively, for each of these
cases). Therefore $X(\Sigma_g^1)$ is a subcomplex of
$\Gamma(\Sigma_g^1) \subset \mathcal{C}(\Sigma_g^1)$.

Let $X(\Sigma_g^0)$ be the image of $X(\Sigma_g^1)$ under the homotopy
equivalence $\Phi$ from Lemma 3.1.2. Consider now the following result in
\cite{BBM}, which is part of a more general result.

\begin{prop}[\cite{BBM}, Proposition 30]
Let $g \geq 1$. The subcomplex $X(\Sigma_g^1)$ is topologically a sphere of
dimension $2g-2$, and, choosing an orientation, the class
$[X(\Sigma_g^1)] \in \widetilde{H}_{2g-2}(\mathcal{C}(\Sigma_g^1); \mathbb{Z})$
is nontrivial.
\end{prop}

In particular, this implies the following result.

\begin{theorem}
Let $g \geq 2$. Then $\pi_{2g-2}(\Gamma(\Sigma_g^0)) \not \cong 0$.
\end{theorem}

\begin{proof}
The class
$[X(\Sigma_g^1)] \in \widetilde{H}_{2g-2}(\mathcal{C}(\Sigma_g^1); \mathbb{Z})$
is nontrivial, so it must also be nontrivial as an element of
$\widetilde{H}_{2g-2}(\Gamma(\Sigma_g^1); \mathbb{Z})$. Then
$\widetilde{H}_{2g-2}(\Gamma(\Sigma_g^1); \mathbb{Z}) \not \cong 0$. The map
$\Phi: \Gamma(\Sigma_g^1) \rightarrow \Gamma(\Sigma_g^0)$ is a homotopy
equivalence, so
$\widetilde{H}_{2g-2}(\Gamma(\Sigma_g^0); \mathbb{Z}) \not \cong 0$. Since
$\Gamma_g^0$ is $(2g-3)$-connected by Theorem 2.4.1, the Hurewicz theorem
implies that $\pi_{2g-2}(\Gamma(\Sigma_g^0)) \not \cong 0$.
\end{proof}

\subsection{Main results}

We are now ready to state the main results of this paper. Combining Theorem
2.4.1, Theorem 3.1.1, and Theorem 3.2.2 proves the following result.

\begin{theorem}
If $\Sigma$ is a genus $g>1$ Heegaard surface for $S^3$, the disk complex
$\Gamma(\Sigma)$is homotopy equivalent to a nontrivial wedge of
$(2g-2)$-spheres.
\end{theorem}

\begin{corollary}
If $\Sigma$ is a genus $g>1$ Heegaard surface for $S^3$, then it is
topologically minimal with index $2g-1$.
\end{corollary}

This result confirms that $S^3$ contains topologically minimal surfaces, so
Bachman's conjecture is false.

\bibliographystyle{hamsplain}
\bibliography{mybibliography}
\end{document}